\documentclass[sn-mathphys,Numbered]{sn-jnl}

\usepackage[utf8]{inputenc}
\usepackage{xcolor}
\usepackage{algpseudocode}
\usepackage{dirtytalk}
\usepackage{algorithm}
\usepackage{amsfonts,amsmath,amssymb,amsthm}
\usepackage{lmodern}
\usepackage[title]{appendix}%
\usepackage{manyfoot}%

\usepackage{hyperref}

\usepackage{tcolorbox}

\newtheorem{theorem}{Theorem}

\newtheorem{proposition}{Proposition}
\newtheorem{remark}{Remark}

\newtheorem{example}{Example}

\begin{document}

\title{From Euler-Jacobi to Bogoyavlensky and back}

\author[1]{Davide Murari}\email{dm2011@cam.ac.uk}
\author[2]{Nicola Sansonetto}\email{nicola.sansonetto@univr.it}

\affil[1]{\orgdiv{Department of Applied Mathematics and Theoretical Physics}, \orgname{University of Cambridige}}

\affil[2]{\orgdiv{Department of Computer Science}, \orgname{University of Verona}}

\abstract{This work focuses on two notions of non-Hamiltonian integrable systems: 
{\it B}-integrability and Euler–Jacobi integrability. We first show that the first notion is stronger. We then investigate which possible ``non-evident'' properties one can add to the Euler–Jacobi Theorem to make the dynamics {\it B}-integrable.}

\maketitle

\section{Introduction}
Integrability has been and still is an active field of research in the Hamiltonian framework (both in finite and infinite dimensions). In contrast, its investigation drove less attention outside the 
Hamiltonian setup, at least until the last decades \cite{kozlov1985integration,hermans1995symmetric,bogoyavlenskij1998extended,bates1999completely,Fedorov1999,zenkov2000dynamics,borisov2002plane,borisov2002surface,fasso2005periodic,fasso2002geometric,borisov2008conservation,bolsinov2011hamiltonization,kozlov2013euler,fasso2015quasi,kozlov2019tensor,balseiro2022first}. For finite-dimensional Hamiltonian systems, the research mainly focused on the notion of complete integrability and Liouville-Arnold Theorem
\cite{arnol2013mathematical,duistermaat1980}, probably in relation with the
fundamental role they play in quantization (see e.g. \cite{woodhouse1992geometric}), semiclassical quantum mechanics (see e.g. \cite{KM1993}) and perturbation theory \cite{arnol2013mathematical}. However, other more generic notions of integrability have been developed, such as the so-called non-commutative integrability introduced independently by Nekhoroshev \cite{nekhoroshev}, and by Mishchenko-Fomenko \cite{MF}, and play as well a crucial role, for example, in Perturbation Theory (for more details, see \cite{fasso2005superintegrable,fasso2022perturbation}) and in quantization \cite{KM1993,sepe2018integrable}. In a non-commutative integrable system, the dynamics is still conjugate 
to a linear one but on isotropic tori, and the phase space 
is endowed with a so-called dual-pair structure\footnote[1]{The 
dual-pair structure named by Weinstein \cite{wenstein1971} 
is also known as bi-foliation or bi-fibration 
\cite{KM1993}: the phase space is foliated/fibrated by isotropic invariant tori, and the isotropic foliation admits a polar (with respect to the symplectic
structure) co-isotropic foliation, which in non-commutative integrable systems is also defined by and invariant for the dynamics.}. 
The Nekhoroshev–Mishchenko–Fomenko~Theorem provides a generalization of the Liouville–Arnold one, reducing to Liouville–Arnold~Theorem if the tori are Lagrangian (in which case the polar foliation coincides with the starting one) and recovers the superintegrable case if the isotropic tori are 1-dimensional (for a detailed discussion on these aspects see, e.g., \cite{FF2005}). 
Generalizations of the Nekhoroshev–Mishchenko–Fomenko Theorem for Hamiltonian systems on Poisson manifolds can be found in \cite{LGMVH}, on almost-symplectic manifolds in \cite{FS2007,SS2012}, and on contact manifolds in \cite{banyaga1999geometry,jovanovic2012noncommutative,khesin2010contact,banyaga2017complete}. To our knowledge, the result in \cite{FS2007} is the first attempt 
to extend the Liouville–Arnold and Nekhoroshev–Mishchenko–Fomenko~Theorems outside 
the Hamiltonian framework. Despite being less studied, the notion of integrability for non-Hamiltonian systems still has several different connotations. Moving away from the Hamiltonian formalism implies not having access to the rich structure of symplectic and Poisson geometry. To compensate for this loss while still being able to conjugate the dynamics to linear flows locally, one needs to require the presence of additional tensor invariants, as can be seen in \cite{bogoyavlenskij1998extended,Fedorov1999}. 
A characterization of integrability for non-Hamiltonian vector fields, named {\it broad} or {\it B}-integrability, is rather recent, and it is due to Bogoyavlensky \cite{bogoyavlenskij1998extended} and, independently, to Fedorov \cite{Fedorov1999}. 
{\it B}-integrability, which reduces to the aforementioned results in the Hamiltonian setup, gives conditions that ensure when a vector field can be (semi-globally) conjugated to a linear flow as long as it admits enough first integrals and dynamical symmetries. 
The lack of an underlining invariant geometric structure or a variational origin of the vector field and hence of the Noether Theorem (see \cite{FS2009}) implies that there are no natural relations between the presence of symmetries and the existence of first integrals, contrary to what occurs in the Hamiltonian case. For a recent and comprehensive analysis of Hamiltonian and non-Hamiltonian integrability for finite dimensional dynamical systems, see \cite{zung2018}.

Nevertheless, there are non-Hamiltonian vector fields 
that do not satisfy the hypotheses of the Bogoyavlensky Theorem but are 
still integrable, that is, their flows can be conjugated to linear ones 
on tori up to a time reparametrization. We stress here that, as proposed by Kozlov \cite{kozlov1985integration,kozlov2013euler,kozlov2019tensor}, the integrability\footnote{We still focus on conjugation to linear flows on tori, however the idea of Kozlov is more general, and refers to integrability by quadrature or, from a more geometric point of view, of the so-called {\it geometric integrability} \cite{sepe2018integrable}.} of a vector field is usually related to the number of tensor invariants characterizing the dynamics. It is typically the case that, at least locally, if there are $n$ tensor invariants for a smooth vector field on an $n$-dimensional manifold, it is integrable. To get semi-global extensions of this result as with {\it B}-integrability, however, the types of tensor invariants that must be present are not entirely arbitrary.

\subsection{The Bogoyavlensky and Euler–Jacobi Theorems}
In this section, we introduce two of the most fundamental results about non-Hamiltonian integrability: the Bogoyavlensky and Euler–Jacobi Theorems, which will be the object of investigation of this paper.

\begin{theorem}[Bogoyavlensky]\label{thm:bogo}
Let $M$ be a smooth $n$-dimensional manifold and $X$ a smooth vector field on $M$. Assume that for an integer $0<k\leq n$ there exists a surjective submersion $$F=(f_1,...,f_{n-k}):M\to \mathbb{R}^{n-k}$$ with compact and connected fibers, such that: 
\begin{itemize}
 \item[B1.] $L_X f_\alpha = 0$, for any $\alpha =1, \ldots, n-k$. 
\end{itemize}
Moreover, assume the existence of $k$ everywhere linearly independent smooth vector fields $Y_1,...,Y_k$ on $M$ such that
\begin{itemize}
\item[B2.] $\mathcal{L}_{Y_i}X = [Y_i,X] = 0$, $i=1,...,k$,
\item[B3.] $[Y_i,Y_j]=0$, $i,j=1,...,k$,
\item[B4.] $\mathcal{L}_{Y_i}f_\alpha = 0$, for all $i=1,...,k$, $\alpha=1,...,n-k$.
\end{itemize}
Then 
\begin{enumerate}
 \item the fibers of $F$ are diffeomorphic to the $k-$dimensional torus $\mathbb{T}^{k}$,
 \item $M$ is a locally trivial fibration in $k$-dimensional tori,
 \item $X$ is conjugated to a vector field with linear flow in a neighborhood of each invariant torus.
\end{enumerate}

\end{theorem}

We will call a pair $(M,X)$, where $M$ is a smooth manifold, and $X$ a smooth vector field on $M$ satisfying the hypotheses of Theorem~\ref{thm:bogo}, \textit{B-integrable} or \textit{broadly-integrable}. 

The notion of {\it B}-integrability is general enough that many non-Hamiltonian systems that manifest quasi-periodic dynamics satisfy its hypothesis. Indeed, from a semi-global point of view, if the flow of a vector field is conjugated to a linear flow on tori, first integrals 
and dynamical symmetries satisfying the hypotheses of Theorem~\ref{thm:bogo} always exists, and thus, Theorem~\ref{thm:bogo} seems more a characterization 
of integrability than a criterion for it. Nevertheless, there exist systems 
that exhibit a quasi-periodic behavior on tori, at least up to a time 
reparametrization, but seem not to be {\it B}-integrable \cite{borisov2002plane,tsiganov2010integrable,gajic2019nonholonomic,garcia2019integrability}. In particular, they cast into the conditions of another theorem, known as the Euler–Jacobi Theorem \cite{arnold1988dynamical,borisov2008conservation}, that guarantees the quasi-periodicity of the motions up to a time reparametrization.

\begin{theorem}[Euler–Jacobi]\label{thm:eulerJac}
Let $X$ be a never-vanishing smooth vector field on an $n$-dimensional smooth manifold $M$, and let 
$F=(f_1,...,f_{n-2}): M \to \mathbb{R}^{n-2}$ be a
surjective submersion with compact and connected fibers such that 
\begin{itemize}
 \item[EJ1.] $\mathcal{L}_{X}f_\alpha = 0$, for $\alpha=1,\ldots,n-2$. 
\end{itemize} 
Let $\mu\in \Omega^n(M)$ be a volume form on $M$ invariant with respect to $X$, i.e. $\mathcal{L}_X\mu=0$. Then:
\begin{itemize}
 \item the level sets of $F$ are diffeomorphic to the 2-dimensional torus, $\mathbb{T}^2$;
 \item $M$ is a locally trivial fibration in $2$-dimensional tori;
 \item if $c\in F(M)$, there exists a neighborhood $U_c\subset F(M)$ of $c$ and a diffeomorphism $\mathcal{C}=(f_1,...,f_{n-2},x_{n-1},x_n):F^{-1}(U_c)\to U_c\times \mathbb{T}^2$, such that in these coordinates $X$ reads
\begin{equation}\label{eq:teoremaeq}
\begin{cases}
\dot{f}_\alpha = 0,\quad \alpha=1,...,n-2,\\
\dot{x}_a = \lambda_a(f)/\Phi,\quad a=n-1,n,
\end{cases}
\end{equation}
for some function $\Phi\in C^\infty(U_c\times \mathbb{T}^2)$, $\Phi=\Phi(F,x_{n-1},x_n)$, and $\lambda_{n-1},\lambda_n\in C^\infty(U_c)$.
\end{itemize} 
\end{theorem}
We will call a triplet $(M,\mu,X)$, where $M$ is a smooth manifold, $\mu$ a volume form, and $X$ a vector field on $M$ satisfying the hypotheses of Theorem~\ref{thm:eulerJac} \textit{EJ-integrable} or {\it integrable in the sense of Euler–Jacobi}. 

In most of the paper, we restrict the study of {\it B}-integrability to the case $k=2$, i.e., the invariant fibers defined by the first integrals are diffeomorphic to $\mathbb{T}^2$, as for {\it EJ}-integrable systems. We consider explicitly the case $k=1$ only in Proposition \ref{pr:newFirstInt}.

\subsection{Aims and contributions of the paper} \label{se:comp}

This paper investigates some relations between the two integrability theorems previously outlined. This comparison is motivated by the similarities in their assumptions and outcomes. At the core of our work lies the question: Are there any \say{non evident}\footnote{Here by \say{non evident} we mean conditions that if added to the statement of the Euler–Jacobi Theorem one obtains a version of the Bogoyavlensky Theorem, i.e., the assumption of a further functionally independent first integral preserved by the dynamical symmetries, or of a further linearly independent dynamical symmetry that preserves the fibration.} 
sufficient conditions we can add to the assumptions of the Euler–Jacobi Theorem to ensure the {\it B}-integrability of a vector field? Before focusing on this question in Section \ref{se:EJImpliesBogo}, we first show in Section \ref{se:bogoImpliesEJ} that {\it B}-integrability is the strongest among the two notions of integrability, which is an easier problem.

\null
Throughout this paper, we will tacitly assume that all objects (functions, maps, and vector and tensor fields) are smooth and that all considered manifolds are smooth, 
orientable, and connected unless stated otherwise. We also assume the $1$-torus, $S^1$, is diffeomorphic to $\mathbb{R}/2\pi\mathbb{Z}$, and the $2$-torus, $\mathbb{T}^2$, is diffeomorphic to $\mathbb{R}^2/(2\pi \mathbb{Z})^2$.

We now introduce the types of invariants we consider throughout the manuscript. Let $X$ be a vector field on the $n$-dimensional manifold $M$. A differential $k$-form $\alpha$ on $M$, $k\leq n$, is $X$-invariant or invariant along the flow of $X$ if $\mathcal{L}_X\alpha =0$, where $\mathcal{L}_X$ denotes the Lie-derivative along $X$. 
We will focus on one-forms and volume forms, i.e., on the cases $k=1$ and $k=n$. A first integral $f$ of $X$ is a function $f:M\to\mathbb{R}$ satisfying $\mathcal{L}_Xf= X(f)=df(X)=0$. A dynamical symmetry of $X$ is a vector field $Y$ on $M$ such that $[X,Y]=0$. A Lie-point symmetry of $X$ is, instead, a vector field $Y$ on $M$ satisfying $[X,Y]=\lambda X$ for a function $\lambda$ on $M$.

\section{Bogoyavlensky implies Euler–Jacobi}\label{se:bogoImpliesEJ}
We now show that a {\it B}-integrable system is also {\it EJ}-integrable. 

\begin{theorem}[{\it B}-integrability implies {\it EJ}-integrability]
Let $(M,X)$ be a {\it B}-integrable system. Then there exists a volume form $\mu$ on $M$ for which $(M,\mu,X)$ is {\it EJ}-integrable.
\end{theorem}

\begin{proof}
Let $Y\in\mathfrak{X}(M)$ be the dynamical symmetry and $F=(f_1,...,f_{n-2}):M\to\mathbb{R}^{n-2}$ the subjective submersion that ensures the {\it B}-integrability of $X$. We now construct an $X$-invariant volume form $\mu$ on $M$. Since $F$ is a surjective submersion with compact and connected fibers, the Ehresmann Theorem ensures that $F$ defines a locally trivial fibration. Let us consider an open covering $\{U_c\}_{c\in\mathcal{I}}$ of the base of the fibration $F(M)$ such 
that $F^{-1}(U_c)\simeq U_c\times F^{-1}(c)\simeq U_c\times \mathbb{T}^2$. 
On the open set $F^{-1}(U_c)$ we can thus consider the coordinates adapted 
to the fibration $(f_1,...,f_{n-2},x,y)$. We define $X_c$ and $Y_c$ as the 
vector fields related to $X$ and $Y$ via the inclusion map $i_c:F^{-1}(U_c)\to M$, and let $\{\alpha_{X_c},\alpha_{Y_c}\}$ denote the dual basis associated to $X_c$ and $Y_c$ over $F^{-1}(U_c)$. We can then construct a volume form $\mu_c$ on 
$F^{-1}(U_c)$ as
\[
\mu_c = df_1 \wedge ... \wedge df_{n-2} \wedge \alpha_{X_c} \wedge \alpha_{Y_c}.
\]
$\mu_c$ is by construction constant along $X$. To get a globally defined volume form over $M$, we consider the partition 
of unity $\{\rho_c\}_{c\in\mathcal{I}}$ subordinate to the open covering $\{U_c\}_{c\in\mathcal{I}}$ of $F(M)$. 
It induces a partition of unity $\{F^*\rho_c\}_{c\in\mathcal{I}}$ of $F^{-1}(F(M))$ subordinate 
to the open covering $\{F^{-1}(U_c)\}_{c\in\mathcal{I}}$. Such a partition of unity allows us to 
define the volume form
\[
\mu = \sum_{c\in\mathcal{I}} (F^*\rho_c) \widetilde{\mu}_c,
\]
where $\widetilde{\mu}_c$ is the extension by zero of $\mu_c$ from $F^{-1}(U_c)$ to $M$. Since for every $z\in M$ and $c\in\mathcal{I}$ one has
\[
\mathcal{L}_X(F^*\rho_c)(z)=X(\rho_c\circ F)(z) = (d\rho_c)_{F(z)}dF_z(X(z)) = (d\rho_c)_{F(z)}(0)=0,
\]
we conclude that $\mu$ is $X-$invariant and hence $X$ is {\it EJ}-integrable.
\end{proof}

\section{From Euler–Jacobi to Bogoyavlensky}\label{se:EJImpliesBogo}

In this section, we derive some results referring to the other implication. We recover some sufficient conditions ensuring that some systems integrable in the sense of Euler–Jacobi are {\it B}-integrable. More precisely, we address the following question: given an $n$-dimensional manifold $M$ and a vector field $X$ on it with $(n-2)$ functionally independent first integrals, is there any relation between the integrability of $X$ with respect to the two approaches?
This question is rather natural because both Theorems \ref{thm:bogo} and \ref{thm:eulerJac} guarantee the quasi-periodicity of the flow on invariant tori, but a time reparametrization of the vector field may be required for the latter. The existence of an invariant volume form for a system with $n-2$ functionally independent first integrals allows us to get the hypotheses of the Euler–Jacobi Theorem, and the integrability comes at the cost of a reparametrization in time. A more precise question is thus if there are any conditions that we can add to {\it EJ}-integrable systems to ensure their {\it B}-integrability. This question has already been studied from the point of view of guaranteeing the existence of an additional non-trivial first integral, i.e., providing the {\it B}-integrability on $S^1$ \cite[e.g.]{llibre2012note}, but the {\it B}-integrability via a linearly independent dynamical symmetry is not very well studied. This is the focus of our derivations. 

Before presenting the results, we introduce three toy examples that justify our interest in the connections between the two theorems since they point out the different situations that might appear.

\begin{remark}{\it The existence of an independent symmetry field for vector fields on $2$-tori has been previously studied, for example, in \cite{sinaui1976introduction}. A condition that guarantees the existence of the symmetry is the presence of a Poincar\'e section with a constant return time. Such condition strongly resonates with the proof of the Euler–Jacobi Theorem in \cite{kozlov2013euler}, where the time-reparametrization is introduced exactly to compensate the fact that the first return map could not be constant.}
\end{remark}

\subsection{Examples}
Let us focus on never-vanishing vector fields over $\mathbb{T}^2$. The examples we provide are of dynamical systems for which it is simple to verify the preservation of the volume form $\mu =dx\wedge dy$, when expressed in local coordinates. We remark that, on $\mathbb{T}^2$, the conservation of a volume form is enough to ensure the {\it EJ}-integrability. The first considered vector field will also be {\it B}-integrable on $\mathbb{T}^2$ 
thanks to an additional dynamical symmetry $Y$, 
the second will be {\it B}-integrable on $S^1$ thanks to the presence of a non-trivial 
first integral, while the third will not be {\it B}-integrable.

\begin{example}[{\it B}-integrable on $\mathbb{T}^2$]
Let us consider the non-vanishing vector field
\[
X(x,y) = \left(\sin(y)+\sqrt{2}\right)\partial_x + \partial_y.
\]
First of all, we notice that
\begin{equation}\label{eq:isEJ}
\mathcal{L}_{X}\left(dx\wedge dy\right) = d\left((\sin(y)+\sqrt{2})dy - dx\right) =0,
\end{equation}
and hence $X$ is {\it EJ}-integrable. We now show that $X$ admits a linearly independent dynamical symmetry $Y\in\mathfrak{X}(\mathbb{T}^2)$, and is hence {\it B}-integrable on $\mathbb{T}^2$. Let us define a generic vector field $Y$ as
\[
Y(x,y) = a(x,y)\partial_x + b(x,y)\partial_y,
\]
for a pair of functions $a,b\in\mathcal{C}^{\infty}(\mathbb{T}^2)$. By enforcing $[X,Y]=0$, we get
\[
0=[X,Y] = \left(\partial_x a (\sin{y}+\sqrt{2}) + \partial_y a -b\cos(y)\right)\partial_x + \left(\partial_x b(\sin{y}+\sqrt{2})+\partial_yb\right)\partial_y.
\]
The only way for the second component to vanish is that $\mathcal{L}_Xb = 0$, i.e., we must have that $b$ is a first integral of $X$. However, this can happen only if $b\in\mathcal{C}^{\infty}(\mathbb{T}^2)$ is a constant function, because $b$ must be $2\pi$-periodic in both its arguments to be well defined on $\mathbb{T}^2$. Let us then suppose $b$ is constantly equal to $k$, for a $k\in\mathbb{R}$. It thus remains to impose
\[
\partial_x a \left(\sin{y}+\sqrt{2}\right) + \partial_y a = k\cos{y},
\]
which can be solved with the methods of characteristics, leading to the solution $a(x,y)=k(\sin(y)+c)$, for a generic $c\in\mathbb{R}$. We thus conclude that, for every $c\in\mathbb{R}\setminus \{\sqrt{2}\}$, the vector field
\[
Y_c(x,y) = \left(\sin{y}+c\right)\partial_x + \partial_y\in T_{(x,y)}\mathbb{T}^2
\]
is a linearly independent dynamical symmetry of $X$, which is hence {\it B}-integrable on $\mathbb{T}^2$.
\end{example}

\begin{example}[{\it B}-integrable on $S^1$]

This non-vanishing vector field provides the second example
\[
X(x,y) = (\sin(y)+2)\partial_x + \sin(x)\partial_x.
\]
By the same calculation as in Example \ref{eq:isEJ}, one can verify that $X$ is {\it EJ}-integrable since $\mathcal{L}_X \mu=0$, with $\mu=dx\wedge dy$. We can also see that $i_X{\mu} = (\sin(y)+2)dy - \sin(x)dx$. Thus, on the covering space $\pi:\mathbb{R}^2\to \mathbb{T}^2$, we have $i_{\widetilde{X}}\widetilde{\mu} = d\widetilde{h}(x,y)$, with $\widetilde{h}:\mathbb{R}^2\to\mathbb{R}$ defined as $\widetilde{h}(x,y)=-\cos(y)+2y+\cos(x)$, $\widetilde{X}$ $\pi$-related with $X$ and $\widetilde{\mu}=\pi^*\mu$. Despite $\widetilde{h}$ not being $2\pi-$periodic, and hence $i_{X}\mu$ not being exact on $\mathbb{T}^2$, this result informs us that $F = \Phi \circ \widetilde{h}$, with $\Phi:\mathbb{R}\to\mathbb{R}$ smooth and $2\pi-$periodic, is conserved by the dynamics on $\mathbb{T}^2$, and it is also in $\mathcal{C}^{\infty}(\mathbb{T}^2)$ because $2\pi-$periodic in both $x$ and $y$. $F$ is thus a non-trivial first integral of $X$, ensuring that $X$ is {\it B}-integrable on $S^1$.
\end{example}

\begin{example}[Not {\it B}-integrable]\label{ex:example3}

In \cite[Example II.12]{perrella2023existence}, the authors show that the vector field $X=fX_0\in\mathfrak{X}(\mathbb{T}^2)$, $f\in\mathcal{C}^{\infty}(\mathbb{T}^2)$, $f>0$, with
\[
X_0(x,y) = a\partial_x + b\partial_y,
\]
$(a,b)\in\mathbb{R}^2$ incommensurable, neither admits linearly independent symmetries nor non-trivial first integrals. The non-vanishing vector field $X$ is thus not {\it B}-integrable. On the other hand, $X=fX_0$ preserves the volume form $\mu = 1/f dx\wedge dy$ since $di_X\mu=0$. Thus, $X$ is {\it EJ}-integrable but not {\it B}-integrable. 
\end{example}

\subsection{Conservation of a differential one-form}
In this subsection, we analyze some sufficient conditions that one should add to the hypotheses of the Euler–Jacobi Theorem to ensure the existence of an independent dynamical symmetry $Y$ of the vector field $X$ and hence the {\it B}-integrability of the system on two-dimensional tori. We remark that a key ingredient in the proof of the Euler–Jacobi Theorem is the fiberwise construction\footnote{This construction is fiberwise with respect to the fibers of the fibration defined by the first integrals.} of a dynamical symmetry $Y$ of a time-reparametrization 
$fX$ of $X$, which is not in general a dynamical symmetry of $X$ itself,
with $f$ a nowhere vanishing function on the torus. Therefore, we wish to investigate possible 
non-evident conditions that added to the hypotheses of the Euler–Jacobi Theorem guarantee the {\it B}-integrability of the system. 

Let $X$ be a vector field on $\mathbb{T}^2$ that, in local coordinates $(x,y)\in S^1\times S^1$, reads:
\begin{equation}\label{eq:vf}
X = X^1(x,y)\partial_x + X^2(x,y)\partial_y.
\end{equation}

Assume that $\mathcal{L}_X(\mu)=0$, with $\mu$ a volume form on the $2$-torus, then we investigate the conditions for the existence of an independent dynamical symmetry of $X$. If the coefficients in \eqref{eq:vf} depend only on one angle, it is easy to find a dynamical symmetry.
\begin{itemize}
 \item If $X^i(x,y) = f^i(x)$, with $f^1$ which is not identically zero, 
 then $Y=\partial_y$ is an independent symmetry for $X$.
 \item If $X^i(x,y) = g^i(y)$, with $g^2$ which is not identically zero, 
 then $Y=\partial_x$ is an independent symmetry for $X$.
\end{itemize}
These vector fields are {\it EJ}-integrable and {\it B}-integrable. 
Apart from these simple cases, we now analyze more general situations.

\begin{proposition}
An {\it EJ}-integrable system $(M,\mu,X)$, with $M$ an $n$-dimensional manifold, is {\it B}-integrable if there exists a one-form $\alpha$ on $M$ constant along the flow of $X$ and one of the two following conditions holds:
\begin{enumerate}
    \item $df_1\wedge ... \wedge df_{n-2}\wedge \alpha\in\Omega^{n-1}(M)$ is not proportional to $i_X\mu$,
    \item there exists $z\in M$ such that $\alpha_z=0$ but $(d\alpha)_z\neq 0$.
\end{enumerate}
\end{proposition}
We remark that the vector field $Y$ with $i_Y\mu=df_1\wedge ... \wedge df_{n-2}\wedge \alpha$ is well defined since $\mu$ is a volume form.
\begin{proof}
Consider the vector field $Y$ on $M$ such that $i_Y\mu = df_1\wedge ... \wedge df_{n-2}\wedge \alpha$. Since
\[
i_{[X,Y]}\mu = i_{\mathcal{L}_X Y}\mu = \mathcal{L}_X(i_Y\mu) - i_Y(\mathcal{L}_X\mu) = \mathcal{L}_X(i_Y\mu) = 0
\]
given that $ \mathcal{L}_X\alpha = 0$ and $\mathcal{L}_X(df_i)=d\mathcal{L}_Xf_i=0$, it follows $[X,Y]=0$. By contradiction, let us assume that there exists a strictly positive function $f$ on $M$ such that $Y=fX$. This implies
\[
df_1\wedge ... \wedge df_{n-2}\wedge \alpha = i_Y \mu = f i_X\mu,
\]
hence contradicting assumption (i). 

Moving to assumption (ii), we assume, by contradiction, that $Y=fX$ with $f$ a strictly positive function on $M$. Thus, $g\,df_1\wedge ... \wedge df_{n-2}\wedge \alpha = i_X\mu$ with $g=1/f$, and hence
\[
di_X\mu =\mathcal{L}_X\mu = 0 = d(df_1\wedge ... \wedge df_{n-2}\wedge (g\alpha)) = (-1)^{n-2}df_1\wedge ... \wedge df_{n-2}\wedge (dg\wedge \alpha + g\,d\alpha).
\]
Let us now focus on the point $z\in M$ with $\alpha_z=0$ but $(d\alpha)_z\neq 0$, where 
\begin{equation}\label{eq:vanishing}
0=(di_X\mu)_z =g(z)(df_1\wedge ... \wedge df_{n-2}\wedge d\alpha)_z(-1)^{n-2}.
\end{equation}
Since $F$ is assumed to be a surjective submersion, the $(n-2)$-differential form $df_1\wedge ... \wedge df_{n-2}$ is never vanishing. Thus, \eqref{eq:vanishing} implies that $g(z)=0$, which is a contradiction. We conclude that $X$ is {\it B}-integrable since it admits a linearly independent dynamical symmetry $Y$.
\end{proof}

\subsection{{\it EJ}-integrable systems, Lie-point symmetries, and {\it B}-integrability}
In this subsection, we relate the notions of {\it B}-integrability, {\it EJ}-integrability, and Lie-point symmetry. We start with a proposition guaranteeing that when a suitable Lie-point symmetry exists, one can also build a dynamical symmetry. This first result can also be seen as a particular instance of {\it B}-integrability, where $X$ has a dynamical symmetry $Z$ preserving the first integrals in $F$ and has a particular structure. We then conclude with a proposition combining the three notions, and ensuring that given a pair of {\it EJ}-integrable systems $(M,\mu,X)$ and $(M,\mu,Y)$, with $Y$ a Lie-point symmetry of $X$, then $(M,X)$ is {\it B}-integrable.
\begin{proposition}
Let $F=(f_1,\ldots,f_{n-2}):M\to\mathbb{R}^{n-2}$ be a surjective submersion with compact and connected fibers, and $X,Y$ be everywhere linearly independent vector fields on $M$ that are tangent to the fibers of $F$. $X$ is {\it B}-integrable if one of the following conditions is satisfied:
\begin{itemize}
 \item[i.] there exist $g,h\in\mathcal{C}^{\infty}(M)$, $h$ never vanishing, such that $[X,Y] = g Y$, and $X(h) = -g h$,
 \item[ii.] there exist $g,h\in\mathcal{C}^{\infty}(M)$, $h$ never vanishing, such that $[X,Y]=g X$, and $X(h)=-g$.
\end{itemize}
\end{proposition}
Before proving the proposition, we remark that the case $[X,Y]=0$, as assumed in the Bogoyavlensky Theorem, is a particular one falling within the intersection of assumptions (i) and (ii), where $X$ is a Lie-point symmetry of $Y$ and vice versa. More explicitly, $[X,Y]=0$ if and only if $g$ is identically zero.
\begin{proof}
\begin{itemize}
 \item[i.] Let us consider the vector field $Z=X+h Y$. We first observe that $Z$ is tangent to the fibers of $F$, and we now prove that it is also a linearly independent dynamical symmetry of $X$. The two vector fields $X$ and $Z$ commute since
\begin{eqnarray*}
[X,X+h\, Y] &= X(h)Y + h[X,Y] = -gh Y + h g Y =0.
\end{eqnarray*}
Since $h$ never vanishes, $Z$ is linearly independent of $X$, which is hence {\it B}-integrable.
\item[ii.] Let us introduce the vector field $Z=h X + Y$
and compute
\[
[X,Z] = (X(h)+g)X = 0.
\]
If $X(h)=-g$, we conclude $[X,Z]=0$, and since $h$ never vanishes, we recover a linearly independent dynamical symmetry. $X$ is thus {\it B}-integrable.
\end{itemize}
\end{proof}

We now present a slightly different result closer to the Euler–Jacobi integrability setup. We consider two {\it EJ}-integrable systems $(M,\mu,X)$ and $(M,\mu,Y)$. We formalize this result in the following proposition.
\begin{proposition}\label{pr:newFirstInt}
 Let $(M,\mu,X)$ and $(M,\mu,Y)$ be two {\it EJ}-integrable systems with the same invariant fibration, and let $X$ and $Y$ be everywhere linearly independent vector fields. Assume there exists a function $\lambda\in\mathcal{C}^{\infty}(M)$ such that $[X,Y]=\lambda X$. Then $X$ is {\it B}-integrable.
\end{proposition}
\begin{proof}
Let $F:M\to\mathbb{R}^{n-2}$ be the invariant fibration for both systems, and $\{U_c\}_{c\in\mathcal{I}}$ be an open covering of $F(M)$ with $F^{-1}(U_c)\simeq U_c\times \mathbb{T}^2$. In $F^{-1}(U_c)$, we can work with the coordinates adapted to the fibration, with respect to which the restriction $\mu_c\in\Omega^n(F^{-1}(U_c))$ of $\mu$, writes
\[
\mu_c = df_1\wedge ... \wedge df_{n-2}\wedge \omega_c,
\]
for a 2-form $\omega_c$ invariant with respect to the vector fields $X_c,Y_c\in\mathfrak{X}(F^{-1}(U_c))$ related respectively with $X,Y\in\mathfrak{X}(M)$, via the inclusion map $i_c:F^{-1}(U_c)\to M$. We then have that
\begin{eqnarray*}
\mathcal{L}_{X_c}(\omega_c(X_c,Y_c)) &= \mathcal{L}_{X_c}(i_{X_c}i_{Y_c}\omega_c) = i_{X_c}\mathcal{L}_{X_c}i_{Y_c}\omega_c + i_{[X_c,X_c]}i_{Y_c}\omega_c \\
&=i_{X_c}i_{Y_c}\mathcal{L}_{X_c}\omega_c + i_{X_c}i_{[X_c,Y_c]}\omega_c \\
&=i_{X_c}i_{[X_c,Y_c]}\omega_c
\end{eqnarray*}
since $\mathcal{L}_{X_c}\omega_c = 0$. Given that $[X_c,Y_c]$ is $i_c$-related with $[X,Y]$, we also have $[X_c,Y_c]=\lambda_c X_c$, where $\lambda_c=i_c^* \lambda\in\mathcal{C}^{\infty}(F^{-1}(U_c))$ since if $z\in F^{-1}(U_c)$ one has
\[
(di_c)_z[X_c,Y_c](z) = [X,Y](i_c(z))= \lambda(i_c(z)) X(i_c(z)) = (di_{c})_z(\lambda_c X_c)(z).
\]
This analysis allows us to conclude that $\omega_c(X_c,Y_c)$ is a first integral of $X_c$. We now extend this result to all of $M$. Let us consider a partition of unity $\{\rho_c\}_{c\in\mathcal{I}}$ subordinate to the open covering $\{U_c\}_{c\in\mathcal{I}}$ of $F(M)$, and define the associated partition of unity $\{F^*\rho_c\}_{c\in\mathcal{I}}$ subordinate to the open covering $\{F^{-1}(U_c)\}_{c\in\mathcal{I}}$ of $F^{-1}(F(M))$. We now consider the function $I\in\mathcal{C}^{\infty}(M)$ defined as
\begin{equation}\label{eq:extendedFI}
I=\sum_{c\in\mathcal{I}} (F^*\rho_c)I_c,
\end{equation}
where the $I_c\in\mathcal{C}^\infty(M)$ are defined as
\[
I_c(m)=\begin{cases}
    \omega_c (X_c,Y_c)(m),\,\,m\in F^{-1}(U_c),\\
    0,\,\,m\in M\setminus F^{-1}(U_c).
\end{cases}
\]
We notice that on the fibers of $F$, $F^{-1}(\{c\})$, the vector field $[\tilde{X}_c,\tilde{Y}_c]$, where $\tilde{X}_c,\tilde{Y}_c$ are related with $X$ and $Y$ via the inclusion $\tilde{i}_c:F^{-1}(\{c\})\to M$, is Hamiltonian with respect to the symplectic form $\tilde{\omega}_c$ related via $\tilde{i}_c$ with $\mu$, where the Hamiltonian function is $\tilde{\omega}_c(\tilde{Y}_c,\tilde{X}_c)$. This can be seen via the derivation
\begin{eqnarray*}
i_{[\tilde{X}_c,\tilde{Y}_c]}\tilde{\omega}_c &= \mathcal{L}_{\tilde{X}_c}i_{\tilde{Y}_c}\tilde{\omega}_c - i_{\tilde{Y}_c}\mathcal{L}_{\tilde{X}_c}\tilde{\omega}_c = \mathcal{L}_{\tilde{X}_c}i_{\tilde{Y}_c}\tilde{\omega}_c \\
&=di_{\tilde{X}_c}i_{\tilde{Y}_c}\tilde{\omega}_c + i_{\tilde{X}_c}di_{\tilde{Y}_c}\tilde{\omega}_c = d\tilde{\omega}_c(\tilde{Y}_c,\tilde{X}_c). 
\end{eqnarray*}
As a consequence, we have that $[\tilde{X}_c,\tilde{Y}_c]=0$ if and only if $\tilde{\omega}_c(\tilde{Y}_c,\tilde{X}_c)$ is constant. We thus conclude that either $\lambda$ is identically zero and hence $Y$ is a dynamical symmetry of $X$, or $I$ in \eqref{eq:extendedFI} is not constant on the joint level set $F^{-1}(\{c\})$ and hence it is not functionally dependent on the other $n-2$ first integrals. Both these options imply the {\it B}-integrability of $X$. If $\lambda$ is identically zero, then $X$ is {\it B}-integrable on $\mathbb{T}^2$. Otherwise, $X$ is {\it B}-integrable on $S^1$.
\end{proof}

\bibliography{biblio}
\end{document}